\documentclass[11pt,a4paper]{amsart}

\usepackage{color}
\usepackage{amsmath,enumerate}
\usepackage{amsfonts}
\usepackage{amssymb}
\usepackage{amsbsy}
\usepackage{amscd}
\usepackage{amsthm}
\usepackage{array,calc,youngtab}
\usepackage{mathrsfs}
\usepackage[english]{babel}
\usepackage{graphicx}
\usepackage[all]{xy}
\usepackage{mathtools}
\usepackage[active]{srcltx}
\usepackage[parfill]{parskip}

\newtheorem{theorem}{Theorem}[section]
\newtheorem{lemma}[theorem]{Lemma}
\newtheorem{definition}[theorem]{Definition}

\newtheorem{proposition}[theorem]{Proposition}
\newtheorem{remark}[theorem]{Remark}
\newtheorem{example}[theorem]{Example}

\newcommand{\TKK}{{\mathrm{Ko}}}
\newcommand{\TT}{{\mathrm{Ti}}}

\newcommand{\tr}{{\mathrm{tr}}}
\newcommand{\eins}{\leavevmode\hbox{\small1\kern-3.8pt\normalsize1}}

\newcommand{\Der}{{\mathrm{Der}}}
\newcommand{\Kan}{{\mathrm{Kan}}}

\newcommand{\mR}{\mathbb{R}}
\newcommand{\mC}{\mathbb{C}}

\newcommand{\mZ}{\mathbb{Z}}

\newcommand{\mD}{\mathbb{D}}

\newcommand{\mg}{\mathfrak{g}}

\newcommand{\fg}{{\mathfrak g}}
\newcommand{\fh}{{\mathfrak h}}

\newcommand{\str}{\mathfrak{str}}
\newcommand{\istr}{\mathfrak{istr}}
\newcommand{\Inn}{{\rm Inn}}

\newcommand{\cD}{\mathcal{D}}

\newcommand{\Span}{\mathrm{Span}}

\newcommand{\cJ}{\mathcal{J}}
\newcommand{\cL}{\mathcal{L}}

\newcommand{\mK}{\mathbb{K}}

\newcommand{\oa}{\bar{0}}
\newcommand{\ob}{\bar{1}}

\newcommand{\End}{{\rm End}}

\newcommand{\id}{{\rm id}}
\newcommand{\ad}{{\rm ad}}

\newcommand{\Out}{{\rm Out}}

\DeclarePairedDelimiter\abs{\lvert}{\rvert}%
\DeclarePairedDelimiter\norm{\lVert}{\rVert}%
\makeatletter
\let\oldabs\abs
\def\abs{\@ifstar{\oldabs}{\oldabs*}}
\let\oldnorm\norm
\def\norm{\@ifstar{\oldnorm}{\oldnorm*}}
\makeatother

\newcommand{\minus}{\scalebox{0.9}{{\rm -}}}
\newcommand{\plus}{\scalebox{0.6}{{\rm+}}}

\begin{document}

\title[Structure and TKK algebras]{On structure and TKK algebras for Jordan superalgebras}

\author{Sigiswald Barbier and Kevin Coulembier}
\date{}

\begin{abstract}
We compare a number of different definitions of structure algebras and TKK constructions for Jordan (super)algebras appearing in the literature. We demonstrate that, for unital superalgebras, all the definitions of the structure algebra and the TKK constructions fall apart into two cases. Moreover, one can be obtained as the Lie superalgebra of superderivations of the other. We also show that, for non-unital superalgebras, more definitions become non-equivalent. As an application, we obtain the corresponding Lie superalgebras for all simple finite dimensional Jordan superalgebras over an algebraically closed field of characteristic zero.
\end{abstract}

\maketitle





\section{Introduction}
There is an acclaimed principle that associates a 3-graded Lie algebra to a Jordan algebra, as developed by Tits, Kantor and Koecher in three variations, see \cite{Tits, Kantor, Koecher}. These three constructions have natural analogues for Jordan superalgebras and some also extend to Jordan (super)pairs. 
The  principle behind these three constructions, and further variations appearing in the literature, is loosely referred to as ``the'' TKK construction.

 A common feature of TKK constructions is that, under the appropriate conditions, they associate {\em simple} Lie superalgebras to {\em simple} Jordan superalgebras or superpairs. They were as such used to classify simple Jordan superalgebras and superpairs, see \cite{Kac, Cantarini, KMZ, Kantor classification, Krutelevich}, but also to study representations of Jordan superalgebras, see \cite{MZ, Shtern, KS}. When the constructions of Tits, Kantor and Koecher are applied to a simple finite dimensional Jordan algebra over the field of complex numbers, they all yield the same Lie algebra, as follows {\it a posteori} from the classification. However, if one applies the TKK constructions to more general algebras, they can give different outcomes.

The aim of the paper is to create more structure in this plethora of TKK constructions, by (dis)proving equivalences of some of the definitions under certain conditions and describing concrete links between the different constructions.
First we consider the
zero component of the 3-graded Lie (super)algebra associated to a Jordan (super)algebra, which is often referred to as the {\em structure algebra}. Then we construct the 3-graded Lie superalgebra out of the structure algebra and the Jordan superalgebra. We refer to this algebra as the {\em TKK algebra}.

We consider four definitions of the structure algebra and show that, for {\em unital} Jordan superalgebras, they lead to two non-equivalent versions of the structure algebra. For {non-unital} Jordan superalgebras, all four definitions are non-equivalent. For completeness, we also review two further definitions of structure algebras of unital Jordan superalgebras, with no direct link to TKK constructions, and prove that these are both equivalent to one of the above definitions. One of these definitions also applies to non-unital Jordan superalgebras, and we prove that it is non-equivalent to the previous definitions.

Then we consider the TKK algebras. First we introduce Kantor's construction. Koecher's construction appears in several forms in the literature, depending on the choice of structure algebra. Finally, the construction by Tits depends on the structure algebra and an auxiliary three-dimensional Lie algebra, which we assume to be $\mathfrak{sl}_{2}$ for now.
This yields 5 definitions of TKK superalgebras associated to a Jordan superalgebra $V$, corresponding to constructions of Tits, Koecher and Kantor:
\begin{center}

\begin{tabular}{|c|c|c|}
\hline
 $\TT(V,\Inn(V),\mathfrak{sl}_{2})$ & Ko$(V)$ & $\Kan(V)$ \\ 
\hline
$\TT(V,\Der(V),\mathfrak{sl}_{2})$ & $\widetilde{\rm Ko}(V)$&\\ 
\hline 
\end{tabular} 
\end{center}
If $V$ is a simple finite dimensional Jordan algebra over the field of complex numbers, it is known that all five Lie algebras are isomorphic. We prove that so long as $V$ is unital, the three Lie superalgebras in the top row are isomorphic. Under the same assumption, the two algebras in the bottom row are then also isomorphic and given by the algebra of derivations of the Lie superalgebras in the top row. For arbitrary $V$, even when finite dimensional, we show that all five algebras can be pairwise non-isomorphic and that the link between bottom and top row through derivations generally fails.

We derive these results for the super case, but they are already pertinent for ordinary Jordan algebras. However, the differences in definitions are more exposed for Jordan superalgebras, as they already appear for {\em finite dimensional simple Jordan superalgebras over the field of complex numbers}. Contrary to simple Lie algebras, simple Lie superalgebras can admit outer derivations, and contrary to Jordan algebras, there is a simple finite dimensional Jordan superalgebra which is non-unital.

Therefore we apply our results to obtain a table with all versions of the TKK construction for the simple finite dimensional Jordan superalgebras over an algebraically closed field of characteristic zero. For this, we can rely on the classification of simple Jordan superalgebras in \cite{Cantarini, Kac} and the calculation of derivations in~\cite{Kac Lie superalgebras, Scheunert}.

We organise the paper as follows. In Section~\ref{SecPrel} we introduce some concepts and terminology regarding Jordan superalgebras and superpairs. In Section~\ref{SecStr} we investigate the different definitions of the structure algebra. In Section~\ref{SecTKK} we compare the constructions of Tits, Kantor and Koecher. In Section~\ref{SecKoe} we study further variations of the Koecher construction, based on the choice of structure algebra. In Section~\ref{SecEx} we use the above to list all the versions of the TKK algebras for the finite dimensional simple Jordan superalgebras over an algebraically closed field of characteristic zero. Finally, in Appendix A we introduce the notation for the Lie superalgebras of type $A$, $P$, $Q$ and $H$ as used in Section~\ref{SecEx}.

\section{Jordan superalgebras and superpairs}
\label{SecPrel}
In the following, we will consider a super vector space $V=V_{\oa} \oplus V_{\ob}$ over a field $\mathbb{K}$. For an element $x$ of $V_i$ we write $|x|=i$ for $i\in\{\oa,\ob\}=\mZ_2$. By $\langle A \rangle$ we will denote the $\mathbb{K}$-linear span of a set $A$. As is customary in the theory of Jordan superalgebras and superpairs, we will always assume that the characteristic of $\mK$ is different from $2$ and $3$. At this stage we make no other assumptions on $\mathbb{K}$. We note furthermore that the main results of section~\ref{SecStr}, \ref{SecTKK} and \ref{SecKoe} still hold if we replace $\mK$ by a ring containing $\frac{1}{2}$ and $\frac{1}{3}$.

\subsection{Jordan superalgebras}
\begin{definition}
A Jordan superalgebra is a super vector space $V$ equipped with a bilinear product which satisfies
\begin{itemize}
\item $V_iV_j \subset V_{i+j}, \qquad i,j \in \mZ_2$
\item $x  y = (-1)^{\abs{x}\abs{y}} y  x$ \qquad (commutativity)
\item $(-1)^{\abs{x}\abs{z}}[L_x,L_{y  z}]+(-1)^{\abs{y}\abs{x}}[L_y,L_{z  x}]+(-1)^{\abs{z}\abs{y}}[L_z,L_{x  y}]=0$ \qquad (Jordan identity),
\end{itemize}
for all homogeneous $x,y,z\in V$.
Here, the operator $L_x\colon V \to V$ is defined by $L_x(y)= xy.$
A Jordan superalgebra $V$ is unital if there exists an element $e \in V$ such that $ex=x=xe$ for all $x \in V$.
\end{definition}
We stress that we do not restrict to finite dimensional algebras.

A Jordan superalgebra satisfies the following relation, see \cite[Section 1.2]{Kac},
\begin{align}\label{Eq: inner derivaiton is a derivation}
[[L_x,L_y],L_z]= L_{x(yz)} -(-1)^{\abs{x}\abs{y}} L_{y(xz)}.
\end{align}

Define the following operators on $V$:
\[ D_{x,y} :=2 L_{x  y} +2 [L_x,L_y].\]

The Jordan triple product is given by \begin{equation}\label{eqTri} \{ x,y,z \}:=D_{x,y}z=2\left((x   y)   z+x   (y  z)  -(-1)^{\abs{x}\abs{y}}y   (x   z) \right). \end{equation}
This triple product satisfies the {\it symmetry property}
 \[ \{ x,y,z \} =(-1)^{\abs{x}\abs{y}+\abs{y}\abs{z}+\abs{x}\abs{z}} \{z,y,x\}, \]
and the {\it 5-linear Jordan identity}
\begin{align} \nonumber
\{ x,y, \{ u,v,w \}\} &- \{ \{ x,y,u\} ,v,w\}\\&=(-1)^{(\abs{x}+\abs{y})(\abs{u}+\abs{v})} (-\{u	,\{v,x,y\},w\} + \{u,v,\{x,y,w\}  \}). \nonumber
\end{align}
The 5-linear identity can be rewritten as
\begin{align}\label{DD}
[D_{x,y},D_{u,v}]&=D_{\{x,y,u\},v} -(-1)^{(\abs{x}+\abs{y})(\abs{u}+\abs{v})}D_{u,\{v,x,y\}} \\
&=D_{x,\{y,u,v\}} -(-1)^{(\abs{x}+\abs{y})(\abs{u}+\abs{v})} D_{\{u,v,x\},y}.\nonumber
\end{align}

\subsection{Jordan superpairs}
A Jordan superpair is a pair of super vector spaces $(V^{\plus}, V^{\minus})$ equipped with two even trilinear products, known as the Jordan triple products, \[ 
\{ \cdot, \cdot, \cdot \}^{\sigma} \colon V^{\sigma} \times V^{\minus\sigma} \times V^{\sigma} \to V^{\sigma},\qquad\mbox{ for } \quad\sigma \in \{+,-\}.
\]
These triple products satisfy symmetry in the outer variables
\begin{align*}
\{x,y,z\}^\sigma &= (-1)^{\abs{x}\abs{y}+\abs{y}\abs{z} + \abs{z}\abs{x}} \{z,y,x\}^\sigma,
\end{align*}
and the  $5$-linear identity
\begin{align*}
\{ x,y, \{ u,v,w \}^\sigma \}^\sigma  &- \{ \{ x,y,u\}^\sigma ,v,w\}^\sigma \\&=(-1)^{(\abs{x}+\abs{y})(\abs{u}+\abs{v})} (-\{u	,\{v,x,y\}^{\minus\sigma},w\}^\sigma + \{u,v,\{x,y,w\}^\sigma  \}^\sigma),
\end{align*}
for homogeneous $x,z,u,w \in V^\sigma$ and $y,v \in V^{\minus\sigma}$.

We will use the following operators 
\[
D_{x,y}\colon V^\sigma \to V^\sigma; \quad z \mapsto \{x,y,z\}^\sigma,
\]
for $x \in V^{\sigma}$ and $y \in V^{\minus\sigma}$.

\begin{example}\label{ExJD}
By the previous subsection, the doubling of a Jordan superalgebra $V$ gives a Jordan superpair $(V^{\plus},V^{\minus}):=(V,V)$ with products $\{x^{\sigma},y^{-\sigma},z^{\sigma}\}^{\sigma}:= \{x,y,z\}$ for $\sigma\in\{+,-\}$. Here we use the notation $x^{\plus}$, resp. $x^{\minus}$, for an element $x\in V$ interpreted as in $V^{\plus}$, resp. $V^{\minus}$. When the context clarifies in which space we interpret $x\in V$, we will leave out the indices.
\end{example}
In the following sections we will often omit the $\sigma$ in the notation for the triple product since it can be derived from the elements it acts on.

\section{Derivations and the structure algebra}\label{SecStr}
In this section we show that the (inner) structure algebra of a unital Jordan superalgebra is isomorphic to the algebra of (inner) derivations of the corresponding superpair. We provide counterexamples to both claims when the Jordan superalgebra is non-unital.

\subsection{The structure algebra for a Jordan superalgebra}
\label{Sect structure algebra}
\begin{definition} Let $V$ be a Jordan superalgebra.
An element $D$ in $\End(V)$ is called a{ \bf derivation} of $V$ if
\[
D(xy) = D(x) y + (-1)^{\abs{x}\abs{D}} x D(y).
\]
We use the notation $\Der(V)$ for the space of derivations, and $\Inn(V)$ for the subspace of {\bf inner derivations}, which is spanned by the operators $[L_x,L_y]$ for $x,y \in V$. 
\end{definition}
The condition on $D\in \End(V)$ to be a derivation is equivalent with 
\begin{equation}
\label{equivDer}[D, L_x] = L_{D(x)}\qquad\mbox{for all $x\in V$. }
\end{equation} Hence equation \eqref{Eq: inner derivaiton is a derivation} implies that $[L_x,L_y]$ is a derivation. One verifies easily that
 $\Der(V)$ is a  subalgebra of $\mathfrak{gl}(V)$. The Jacobi identity on $\mathfrak{gl}(V)$ combined with equation \eqref{equivDer}, for any derivation~$D$, implies that $\Inn(V)$ is an ideal in $\Der(V)$.

We will use the following definition for the structure algebra of Jordan superalgebras, since this is the one that will be required for the Kantor functor. There exist other definitions of the structure algebra in the literature which are not immediately connected to TKK constructions. We will review them in Section~\ref{Section alternative structure algebra} and show that for unital Jordan superalgebras they are all equivalent to our definition.

\begin{definition} \label{Definition structure algebra}
The structure algebra $\str(V)$ is a subalgebra of  $\mathfrak{gl}(V)$, defined as 
\[
\str(V)= \{ L_x \mid x \in V \} + \Der(V).
\]
\end{definition}

\begin{definition} \label{Definition inner structure algebra} 
The inner structure algebra $\istr(V)$ is a subalgebra of  $\mathfrak{gl}(V)$, defined as 
\begin{align*}
\istr(V)&=\{ L_x   \mid x \in V \} + \Inn(V) \\
&=\langle L_x, [L_x,L_y] \mid x,y \in V \rangle.
\end{align*}
\end{definition}
By the above, $\istr(V)$ is an ideal in $\str(V)$.

\begin{remark} \label{Remark decomposition structure algebra for unital Jordan algebras} {\rm For a unital Jordan superalgebra the sum in Definitions~\ref{Definition structure algebra} and~\ref{Definition inner structure algebra}  is a direct sum of super vector spaces:   \[ \str(V)=\{ L_x \mid x \in V\} \oplus \Der(V) \;\text{ and }\; \istr(V)=\{ L_x \mid x \in V\} \oplus \Inn(V), \] since $D(e)=0$ for all $D$ in $\Der(V)$, while $L_x(e)=x$. For non-unital Jordan superalgebras the sums are not necessarily direct, as follows from Example \ref{Example sum not direct} and Remark \ref{Rem: innerstructure and inner derivations}.}
\end{remark}
\begin{example}\label{Example sum not direct}
Consider the commutative three dimensional algebra $V=\langle e_1,e_1,e_3\rangle$ with product \[
e_1^2=e_1, \quad e_1e_2=\frac{1}{2}e_2, \quad e_2^2= e_3,
\]
and all other products of basis elements zero. This is the Jordan algebra $\mathcal{J}_{19}$ in \cite[Section~3.3.3]{KM}. Because \[
[L_{e_1},L_{e_2}]=-\frac{1}{2} L_{e_2},\] we conclude that $L_{e_2}$ is an element of $\Inn(V)$.
\end{example}

\subsection{Derivations of Jordan superpairs}
\begin{definition}
Let $(V^{\plus}, V^{\minus})$ be a Jordan superpair. An element $\mD=(D^{\plus},D^{\minus}) \in \End(V^{\plus})\oplus \End(V^{\minus})$ is called a {\bf derivation} of $(V^{\plus},V^{\minus})$ if
\[
D^{\sigma}(\{x,y,z\})= \{D^{\sigma}(x), y,z\} + (-1)^{\abs{x}\abs{D^{\minus\sigma}}} \{x, D^{\minus \sigma}(y), z\} +(-1)^{(\abs{x}+\abs{y})\abs{D^{\sigma}}} \{x,y, D^{\sigma} (z)\}.
\]
We use the notation $\Der(V^{\plus},V^{\minus})$ for the space of all derivations of $(V^{\plus},V^{\minus})$ and the notation $\Inn(V^{\plus}, V^{\minus})$ for the subspace of {\bf inner derivations}, which is spanned by the operators  $$\mD_{x,y}:=(D_{x,y}, -(-1)^{\abs{x}\abs{y}} D_{y,x}),\quad\mbox{for}\quad x \in V^{\plus}, y \in V^{\minus}.$$
\end{definition}

Observe that any derivation $(D^{\plus},D^{\minus})$ can be written as the sum of derivations where $D^{\plus}$ and $D^{\minus}$ have the same parity. The space $\Der(V^{\plus},V^{\minus})$ hence inherits a grading from the super vector space $\End(V^{\plus})\oplus \End(V^{\minus})$.

 By construction, the space $\Der(V^{\plus},V^{\minus})$ is a subalgebra of $\mathfrak{gl}(V^{\plus}) \oplus \mathfrak{gl}(V^{\minus})$.
The operator $\mD=(D^{\plus},D^{\minus}) \in \End(V^{\plus})\oplus \End(V^{\minus})$ is a derivation if and only if \begin{align} \label{Condition derivation} [D^{\sigma},D_{x,y}]=D_{D^{\sigma}(x),y}+ (-1)^{|\mD||x|} D_{x,D^{\minus\sigma}(y)}.
\end{align}

One can then easily verify that $\Inn(V^{\plus},V^{\minus})$ is an ideal in $\Der(V^{\plus}, V^{\minus})$.
\subsection{Connections}\label{SecISTR}

The main result of this section is the following connection between the structure algebra of a unital Jordan superalgebra and the derivations of the associated Jordan superpair in Example \ref{ExJD}.

\begin{proposition}\label{inner structure is isomorphic to inner derivations}
For a unital Jordan superalgebra $V$ we have
\begin{enumerate}
\item $\str(V) \cong \Der(V,V)$, \label{part one of propoistion structure algebra and derivations}
\item $\istr(V) \cong \Inn(V,V)$.\label{part two of proposition structure algebras and derivations}
\end{enumerate}
\end{proposition}
\begin{remark} {\rm
For a unital Jordan algebra, we have that $\str(V)$ is the Lie algebra of the structure group (Section~\ref{Section alternative structure algebra}) and that $\Der(V,V)$ is the Lie algebra of the automorphism group of the Jordan pair $(V,V)$, (\cite[I.1.4]{Loos}. Since the structure group is isomorphic to the automorphism group of the Jordan pair, \cite[Proposition 1.8]{Loos}, we can immediately conclude that $\str(V)\cong \Der(V,V)$.}
\end{remark}
\begin{remark}\label{RemInn}{\rm Both parts of the proposition do not extend, as stated, to \emph{non-unital} Jordan superalgebras. 
As a counterexample consider again Example \ref{Example sum not direct}. One can easily check that \[\istr(V)= \langle D_{x,y} \mid x,y \in V \rangle  = \langle L_{e_1}, L_{e_2} \rangle \text{ and } \Inn(V,V) = \langle (L_{e_1}, -L_{e_1}),(L_{e_2},0), (0,L_{e_2}) \rangle . \] 
Define $A \in \End(V)$ by $A(e_1)=0$, $A(e_2)=2e_2$ and $A(e_3)=e_3$.
Then we also obtain \[ \str(V)= \istr(V) + \langle A\rangle \text { and } \Der(V,V)=  \Inn(V,V) + \langle (A,A), (A,-A)\rangle .\] 
Subsection~\ref{sec: example Jordan superalgebra K} also contains a counterexample where $\istr(V)\cong \Inn(V,V) $ but $\str(V) \not\cong \Der(V,V)$.  
 }
\end{remark}

Even without the existence of a multiplicative identity $e$, we still have a chain of inclusions if $L_x$ is not a derivation for any $x$ in $V$.
\begin{proposition} \label{Prop: chain of inclusions}
For a Jordan superalgebra $V$ for which $L_x \not\in \Der(V),$ for all $x$ in $V$, we have
\[ \Inn(V,V)\;\subset\;\istr(V)\;\subset\;\str(V)\;\subset\; \Der(V,V).\]
\end{proposition}

\begin{remark}{\rm
 Examples where the second inclusion is strict can be found in Subsection~\ref{section: Examples of TKK constructions} while an example for the third inclusion to be strict can be found in Subsection~\ref{sec: example Jordan superalgebra K}. }
 \end{remark}

Now we start the proofs of the propositions.

\begin{lemma} \label{Lemma: left multiplication and D are derivations}
Let $V$ be a Jordan superalgebra. For $x$ in $V$  and $D$ in $\Der(V)$, we have that 
\[
(L_x,-L_x) \qquad \text{ and } \qquad (D,D) 
\]
are elements of  $\Der(V,V).$
\end{lemma}
\begin{proof}
Using the Jordan identity and equation \eqref{Eq: inner derivaiton is a derivation}, we get for $x,y,z$ in $V$
\begin{align*}
[L_x, D_{y,z}]&=2[L_x,L_{yz}] + 2 [L_x,[L_y,L_z]] \\
&=-2(-1)^{\abs{x}(\abs{y}+\abs{z})} [L_y,L_{zx}]-2(-1)^{\abs{z}(\abs{x}+\abs{y})} [L_z,L_{xy}]+2L_{(xy)z}-2(-1)^{\abs{x}\abs{y}}L_{y(xz)} \\
&= D_{L_x(y),z} -(-1)^{\abs{x}\abs{y}} D_{y, L_x(z)}.
\end{align*}
Thus $(L_x,-L_x)$  satisfies equation \eqref{Condition derivation} and hence belongs to $\Der(V,V)$.
 
Let $D \in \Der(V)$. By equation~\eqref{equivDer}, the Jacobi identity and the definition of $\Der(V)$, we find
\begin{align*}
[D,D_{x,y}] &= 2[D, L_{xy}] + 2 [D,[L_x,L_y]] \\
&=2 L_{D(xy)} -2(-1)^{\abs{D}(\abs{x}+\abs{y})}   [L_x,[L_y,D]] - 2(-1)^{\abs{y}(\abs{D}+\abs{x})} [L_y,[D,L_x]] \\
&= 2 L_{D(x)y} + 2 [L_{D(x)}, L_y] +2(-1)^{\abs{D}\abs{x}} L_{xD(y)} + 2(-1)^{\abs{x}\abs{D}} [L_x, L_{D(y)}] \\
&= D_{D(x),y} + (-1)^{\abs{x}\abs{D}} D_{x,D(y)}. 
\end{align*}
Therefore, also $(D,D)$ satisfies equation \eqref{Condition derivation} and is thus an element of $\Der(V,V)$. 
\end{proof}
\begin{proof}[Proof of Proposition~\ref{Prop: chain of inclusions}]
Since $D_{x,y}=2L_{xy}+2[L_x,L_y]$, the map \begin{align} \label{equation psi}
 \psi\colon \Inn(V,V) \to \istr(V) ;\quad \mD_{x,y}=(D_{x,y},-(-1)^{\abs{x}\abs{y}}D_{y,x}) \mapsto D_{x,y} 
 \end{align}
  is well-defined and clearly a Lie superalgebra morphism. 
Assume $D_{x,y}=0$. Then $L_{xy}=~-[L_x,L_y]$ is a derivation. So by our assumption $L_{xy}=0$, an thus also $D_{y,x}=0$. Therefore $\psi$ is injective.

From the definitions it follows immediately that $\istr(V) \subset \str(V)$.  By assumption, $\str(V)$ is a direct sum of $\{ L_x \mid x \in V \}$ and $\Der(V)$. Together with Lemma \ref{Lemma: left multiplication and D are derivations} this implies that the map
\begin{align} \label{equation phi}
\phi \colon \str(V) \to \Der(V,V); \quad  L_x + D \mapsto (L_x+D, -L_x +D),
\end{align}
is well-defined. This map is clearly injective. The fact that this is also a Lie superalgebra morphism follows from a direct computation, which finishes the proof.
\end{proof}

The rest of this section is devoted to the proof of Proposition~\ref{inner structure is isomorphic to inner derivations}, so we consider a unital Jordan superalgebra $V$. From Remark \ref{Remark decomposition structure algebra for unital Jordan algebras} it follows then that the assumption of Proposition~\ref{Prop: chain of inclusions} is satisfied, so we can use that result.   We will also use the following immediate consequences of equation \eqref{eqTri},
\begin{equation}
\label{eqe}
\frac{1}{2} \{x,e,y\}\,=\, xy\,=\,L_x(y) \qquad \text{ and }\qquad \frac{1}{2} \{e,x,e \}=x.
\end{equation}
Consider the map $\sigma$
\begin{align*}
\sigma\colon   \Der(V,V) \to \Der(V,V); \quad (D^{\plus}, D^{\minus}) \mapsto (D^{\minus}, D^{\plus}).
\end{align*}
Then $\sigma^2 = \id$ and $\Der(V,V)$ decomposes in two subspaces
\begin{align*}
\mathfrak{h} := \{ \mD \in \Der(V,V) \mid \sigma(\mD)= \mD \} \text{ and }
\mathfrak{q}:= \{ \mD \in \Der(V,V) \mid \sigma(\mD)= -\mD \}.
\end{align*}

\begin{lemma}\label{h is isomorphic to subset of der(V,V)}
We have 
\[
 \mathfrak{h}=\{ \mD \in \Der(V,V) \mid D^{\minus}(e) =0 \}. 
\]
\end{lemma}
\begin{proof}
Assume first that $D^{\minus}(e) =0$. By equation~\eqref{eqe}, we have
\[ D^{\minus}(e) = \frac{1}{2}D^{\minus} \{e,e,e\}  = \frac{1}{2} \{ e, D^{\plus} (e), e \} = D^{\plus}(e), \]
and hence also $D^{\plus}(e)=0$. Then we also get for all $x$ in $V$
\[
2 D^{\plus}(x) = D^{\plus} \{e,x,e \}= \{ e, D^{\minus}(x) , e \} = 2 D^{\minus} (x).  
\]
Hence $D^{\plus}=D^{\minus}$. 

Now assume that $D^{\plus}=D^{\minus}$. Equation~\eqref{eqe} then implies \[ D^{\minus} (e)=  \frac{1}{2}D^{\minus} \{e,e,e\} =\frac{1}{2} \{ D^{\minus} (e),  e, e \}+\frac{1}{2} \{ e, D^{\plus} (e), e \}+\frac{1}{2} \{ e,e, D^{\minus}(e) \} =3 D^{\minus}(e).
\]
Hence $D^{\minus}(e)=0$. This concludes the proof.
\end{proof}

\begin{lemma}\label{LemP1}
We have a Lie superalgebra isomorphism
\[
\Der(V) \cong  \mathfrak{h},\qquad \mbox{given by}\quad \phi\colon \Der(V) \to \mathfrak{h} ; \;D \mapsto (D,  D).
\]
\end{lemma}
\begin{proof}

The map $\phi$ is a restriction to $\Der(V)$ of the morphism $\phi\colon \str (V) \to \Der(V,V)$ defined in \eqref{equation phi}. From there we know that it is  injective. The image of $\phi$ is clearly contained in $\mathfrak{h}$.
To show that $\phi$ is surjective, we let $\mD=(D^{\plus},D^{\minus})$ be an element of $\Der(V^{\plus},V^{\minus})$  with $D^{\plus} = D^{\minus},$ i.e. $\mD  \in \mathfrak{h}$. Then $D^{\minus}(e) =0$  by Lemma \ref{h is isomorphic to subset of der(V,V)}.  Hence, using equation \eqref{eqe},
\begin{align*}
D^{\plus}  (xy) & = \frac{1}{2} D^{\plus} \{x,e,y\}  \\&= \frac{1}{2} \{ D^{\plus}x, e, y \} +(-1)^{\abs{D}\abs{x}}  \frac{1}{2} \{ x, D^{\minus}e, y \}+(-1)^{\abs{D}\abs{x}} \frac{1}{2} \{ x, e,D^{\plus} y \}  \\&= D^{\plus}(x)y + (-1)^{\abs{D}\abs{x}} xD^{\plus}(y).
\end{align*}
We conclude that $D^{\plus}=D^{\minus}$ is an element of $\Der(V)$, so $(D^{\plus},D^{\minus})$ is in the image of $\phi$.
\end{proof}

\begin{lemma}\label{LemP2}
We have an isomorphism of super vector spaces
\[ \{ L_x \mid x \in V \} \,\,\,\tilde{\to}\,\,\,\mathfrak{q},\quad\mbox{ given by }\quad L_x \mapsto (L_x, -L_x).\]
\end{lemma}
\begin{proof}
The assignment $L_a \to (L_a,-L_a)$ is a restriction to $\{L_x \mid x \in V \}$ of the injective morphism $\phi\colon \str(V) \to \Der(V,V)$ considered in \eqref{equation phi}. Its image is clearly contained in $\mathfrak{q}$. So the map is well-defined and injective. 
For an element $\mD=(D,-D)$ in $\mathfrak{q}$ we claim that $(L_{D(e)}, -L_{D(e)})= \mD$.
Indeed, using equation~\eqref{eqe}, we have \begin{align*}
 D(x)&= \frac{1}{2}D(\{e,x,e\}) =\frac{1}{2}\{D(e),x,e\}-\frac{1}{2}\{e,D(x),e\}+(-1)^{\abs{D}\abs{x}}\frac{1}{2}\{e,x,D(e)\} \\
 &=2D(e) x -D(x),
 \end{align*}
which implies that $D(x)= L_{D(e)} (x)$. This proves surjectivity. 
\end{proof}

\begin{proof}[Proof of Proposition~\ref{inner structure is isomorphic to inner derivations}]
Consider the injective morphism \[ \phi\colon\Der(V) \oplus  \{ L_x \mid x \in V \}\to \Der(V,V); \qquad L_x + D \mapsto (L_x +D, -L_x+D)\]
of \eqref{equation phi}. From Lemmata \ref{LemP1} and \ref{LemP2} it follows that $\phi$ is also surjective.
This proves part \eqref{part one of propoistion structure algebra and derivations} of the proposition.
For part \eqref{part two of proposition structure algebras and derivations}, consider the injective morphism
\begin{align*}
\psi\colon \Inn(V,V) \to \istr(V); \quad (D_{x,y}, -(-1)^{\abs{x}\abs{y}} D_{y,x} ) \mapsto D_{x,y}.
\end{align*}
defined in \eqref{equation psi}.
Since
\[
L_x= \frac{1}{2} D_{x,e}\; \text{ and }\;
[L_x,L_y] =\frac{1}{4}( D_{x,y},-(-1)^{\abs{x}\abs{y}} D_{y,x}),
\]
the map $\psi$ is surjective, which concludes the proof.
\end{proof}

\subsection{Alternative definitions for the (inner) structure algebra}
\label{Section alternative structure algebra} We review some further definitions appearing in the literature.
Set 
\[ U_{x,y}\colon V \to V; \;z\mapsto  (-1)^{\abs{y}\abs{z}} \{x,z,y\}.\] 
Then for a unital Jordan algebra we define, see \cite[Section 3.1]{Garcia},
\[
\widetilde{\str}(V):= \{ X \in \mathfrak{gl}(V) \mid U_{X(a),b} + (-1)^{\abs{X}\abs{b}} U_{a,X(b)}= X U_{a,b} + (-1)^{\abs{X}(\abs{a}+\abs{b})} U_{a,b} X^\ast \text{ for all } a,b \in V\},
\]
where $X^\ast= -X +2L_{X(e)}$.
In the non-super case, this algebra is the Lie algebra of the structure group, see \cite[Section 9]{Jacobson}.

In the literature we did not find an explicit definition of the structure algebra for the non-unital case using this approach. However, we will define a natural generalization which for a unital Jordan superalgebra will reduce to $\widetilde{\str}(V)$.
So, for $V$ a Jordan superalgebra, define $\str_w(V)$ as the Lie superalgebra consisting of the elements $(X, Y) \in \mathfrak{gl}(V)\oplus \mathfrak{gl}(V)^{op}$ for which 
\begin{align}\label{condition weakly structure algebra}
  &U_{X(a),b} + (-1)^{\abs{X}\abs{a}} U_{a,X(b)}= X U_{a,b} + (-1)^{\abs{Y}(\abs{a}+\abs{b})} U_{a,b} Y   \text{ and } \\   & U_{Y(a),b} + (-1)^{\abs{Y}\abs{a}} U_{a,Y(b)}= Y U_{a,b} + (-1)^{\abs{X}(\abs{a}+\abs{b})} U_{a,b} X  \nonumber
\end{align}
  holds for all  $a,b$ in $V$. If $V$ is a Jordan algebra, one can check that $\str_w(V)$ is the Lie algebra of the group consisting of pairs of `weakly structural transformations', as defined in \cite[II.18.2]{McCrimmon}. 

Using the equality $U_{x,y} (z)= (-1)^{\abs{y}\abs{z}} D_{x,z}(y)$, one finds that the defining conditions of $\str_w(V)$ are equivalent with $(X,-Y) \in \Der(V,V)$.
So we conclude that $\str_w(V) \cong \Der(V,V)$ in full generality.
\begin{lemma}
For a unital Jordan superalgebra $V$, we have
$$\str_w(V)\;\cong\; \widetilde{\str}(V)\;\cong\; \str(V).$$
\end{lemma}
\begin{proof}
Let $X$ be an element of $\widetilde{\str}(V)$. Note that by definition 
$X$ satisfies equation \eqref{condition weakly structure algebra} for $Y=-X+2L_{X(e)}$.  From Lemma \ref{Lemma: left multiplication and D are derivations} we know that $(L_{X(e)}, -L_{X(e)}) \in \Der(V,V)$. Combining this, one shows easily that\[  U_{Y(a),b} + (-1)^{\abs{Y}\abs{b}} U_{a,X(b)}= Y U_{a,b} + (-1)^{\abs{X}(\abs{a}+\abs{b})} U_{a,b} X \text{ for all } a,b \in V\] holds for $Y= -X + 2L_{X(e)}$. Thus the map $\varphi\colon\widetilde{\str}(V) \to \str_w(V); \; X \mapsto  (X , -X + 2L_{X(e)})$ is well-defined. Let $(X,Y) \in \str_w(V)$. Then setting $a$ and $b$ equal to the unit $e$ in equation  \eqref{condition weakly structure algebra} gives us $Y= -X + 2L_{X(e)}$, hence $\varphi$ is an isomorphism.
Since $\str_w(V) \cong \Der(V,V)$, Proposition~\ref{inner structure is isomorphic to inner derivations}, immediately implies that $\str_w(V)$ is also isomorphic to $\str(V)$.
 \end{proof} 

The inner structure algebra is also often defined as the Lie superalgebra spanned by the operators $D_{x,y}\in\End(V)$, see for example \cite[Section 9]{Jacobson}, \cite[Chapter 4]{Springer} and \cite[Section 3.1]{Garcia}. For this algebra we will use the notation \[\widetilde{\istr}(V):= \langle D_{x,y} \mid x,y \in V \rangle.\] 

\begin{lemma}
For a Jordan superalgebra $V$ for which $L_x \not\in \Der(V),$ for all $x$ in $V$, we have
\[
\widetilde{\istr}(V) \cong \Inn(V,V).
\]
In particular, for a unital Jordan superalgebra, we have
\[
\widetilde{\istr}(V) \cong \istr(V).
\]
\end{lemma}
\begin{proof}
By assumption $L_{xy} \not\in \Der(V)$, so we have that  $D_{x,y}=0$ implies $D_{y,x}=0$. Hence the map $\Inn(V,V) \to \widetilde{\istr}(V);\; (D_{x,y},-(-1)^{\abs{x}\abs{y}}D_{y,x}) \mapsto D_{x,y}$ is bijective. It is also clearly an algebra morphism. This proves the first part of the lemma. Since unital Jordan superalgebras satisfy $L_x \not\in \Der(v)$, for all $x$ in $V$, Proposition~\ref{inner structure is isomorphic to inner derivations} immediately implies the second part of the lemma. 
\end{proof}
\begin{remark} \label{Rem: innerstructure and inner derivations} {\rm
In the non-unital case we can both have $\widetilde{\istr}(V) \not\cong \istr(V)$ and $\widetilde{\istr}(V) \not\cong \Inn(V,V)$. The example in Remark \ref{RemInn} is a counterexample for the second part, while a counterexample for the first part is as follows. 
Consider $V:=t\mathbb{K}[t]/(t^k)$, the algebra of polynomials in the variable $t$ without constant term, modulo the ideal $(t^k)= t^k\mathbb{K}[t]$ of polynomials without term in degree lower than~$k$ for some~$k\in\mZ_{>2}$. This is an (associative) Jordan algebra, for the standard multiplication of polynomials, which does not have multiplicative identity. In this example we have $D_{f,g}=2L_{fg}$, for all $f,g\in V$. We hence find that
$$\Inn(V,V)\cong \widetilde{\istr}(V) = \Span_{\mathbb{K}}\{L_{t^2}, L_{t^3},\ldots,L_{t^{k-2}}\}.$$
On the other hand, by definition, we have
$$\istr(V)=\Span_{\mathbb{K}}\{L_{t}, L_{t^2},\ldots,L_{t^{k-2}}\}.$$
As the dimensions of both abelian Lie algebras do not agree, we find $\istr (V) \not\cong \widetilde{\istr}(V)$.
Observe further that $L_{t^{k-2}}$ is an element of $\Der(V)$. Therefore the structure algebra $\str(V)$ also does not have a direct sum decomposition as in Remark \ref{Remark decomposition structure algebra for unital Jordan algebras}
}
\end{remark}

\section{The Tits-Kantor-Koecher construction} \label{SecTKK}
In this section, we will study the three different TKK constructions, dating back to Tits, Kantor and Koecher, and show that, for unital Jordan superalgebras, they are equivalent. Again this claim does not extend to non-unital cases.

\subsection{TKK for Jordan superalgebras (Kantor's approach) }  \label{section TKK for Jordan superalgebras}

In \cite{Kac}, Kac uses the ``Kantor functor" $\Kan$ to classify simple finite dimensional Jordan superalgebras over an algebraically closed field of characteristic zero. This functor is a generalisation to the supercase of the one considered by Kantor in \cite{Kantor}.  In particular this functor provides a TKK construction, which we review for arbitrary Jordan superalgebras over arbitrary fields.

We associate to a Jordan superalgebra $V$, the 3-graded Lie superalgebra 
$$\Kan(V):=\fg=\mg_-\oplus \mg_0\oplus \mg_+,\quad\mbox{with}$$ 
$$\mg_-:=V\quad \mbox{and }\quad \mg_0:= \istr(V)= \langle L_a, [L_a,L_b]\rangle\,\subset \, \End(\fg_-).$$
Finally, $ \mg_+$ is defined as the subspace $\langle P,[L_a,P] \rangle$ of $ \End(\fg_-\otimes \fg_-,\fg_-),$ with
$$P(x,y)=xy\quad\mbox{ and}\qquad [L_a,P](x,y):=a(xy)-(ax)y - (-1)^{\abs{x}\abs{y}} (ay)x.$$
Note that  $P=-[L_e,P]$ for a unital Jordan superalgebra.

As the notation suggests, $[L_a,P]$ corresponds to the superbracket of $L_a\in\fg_0$ and $P\in \fg_+$. The Lie superbracket is then completely defined by 
\begin{itemize}
\item $[\mg_-,\mg_-]=0=[\mg_+,\mg_+]$.
\item $[a,x]=a(x)$, for $a \in \mg_0, x \in \mg_-$.
\item $[A,x](y) = A(x,y)$, for $A \in \mg_+, x,y \in \mg_-$.
\item $[a,B](x,y)= a(B(x,y))-(-1)^{\abs{a}\abs{B}} B(a(x),y) )-(-1)^{\abs{a}\abs{B}+\abs{x}\abs{y}}B(a(y),x)$, for $a \in \mg_0$, $B \in \mg_+$ and $x,y\in\fg_-$.
\end{itemize}
To verify that $\Kan(V)$ is a Lie superalgebra, one can use the following relations (see Proposition~5.1 in \cite{Cantarini})
\begin{itemize}
\item $[P,x]=L_x$
\item  $[[L_a,P],x]=[L_a,L_x]-L_{a  x}$
\item $[L_a,[L_b,P]]=-[L_{a  b},P]$
\item $[[L_a,L_b],P]=0$
\item $[[L_a,L_b],[L_c,P]]=(-1)^{\abs{b}\abs{c}}[L_{a  (c  b) -(a  c)  b} ,P]$.
\end{itemize}

\subsection{TKK for Jordan superpairs (Koecher's approach)} \label{section TKK for Jordan superpairs} 
In \cite{Koecher}, Koecher defined a product on a triple consisting of two vector spaces and a Lie algebra acting on these vector spaces. This product makes the triple into a $3$-graded anti-commutative algebra, which is a Lie algebra if and only if the vector spaces form a Jordan pair and the Lie algebra acts by derivations on the vector spaces. Hence the Koecher construction gives rise to a TKK construction, not only for Jordan algebras, but for Jordan pairs, which is the most natural formulation. Note that, as the concept of Jordan pairs was not yet studied at the time, Koecher did not use this terminology. This TKK construction can be generalised to the supercase, which was for example used by Krutelevich to classify simple finite dimensional Jordan superpairs over an algebraically closed field in characteristic zero in \cite{Krutelevich}.

We associate a 3-graded Lie superalgebra $\TKK(V^{\plus},V^{\minus})$ to the Jordan superpair $(V^{\plus},V^{\minus})$ in the following way.
As vector spaces we have \[\TKK(V^{\plus},V^{\minus})= V^{\plus} \oplus \Inn(V^{\plus},V^{\minus}) \oplus V^{\minus}.\]
The Lie super bracket on $\TKK(V^{\plus},V^{\minus})$ is defined by
\begin{align*}
[x,u]&=\mD_{x,u} \\
[\mD_{x,u},y]&=\mD_{x,u}(y)=\{x,u,y\} \\
[\mD_{x,u},v]&=\mD_{x,u}(v)=-(-1)^{\abs{x}\abs{u}} \{u,x,v\} \\
[\mD_{x,u},\mD_{y,v}]&=\mD_{\mD_{x,u}(y),v} + (-1)^{(\abs{x}+\abs{u})\abs{y}} \mD_{y,\mD_{x,u}(v)} \\
[x,y]&=[u,v]=0,
\end{align*}
for $x,y \in V^{\plus}$, $u,v \in V^{\minus}$.
Recall that $\mD_{x,u}=(D_{x,u}, -(-1)^{\abs{x}\abs{u}} D_{u,x})\in\Inn(V^+,V^-)$.

In case $V$ is a Jordan superalgebra, we simply write $\TKK(V)$ for $\TKK(V,V)$.

Conversely, with each 3-graded Lie superalgebra $\mg=\mg_{-1} \oplus \mg_0 \oplus \mg_{+1}$ 
 we can associate a  Jordan superpair by $\mathcal{J}(\mg)= (\mg_{+1},\mg_{-1})$ with the Jordan triple product given by
\[
\{x^\sigma,y^{-\sigma},z^\sigma \}^\sigma := [[x^\sigma,y^{-\sigma}],z^\sigma].
\]

\begin{definition}
A 3-graded Lie superalgebra $\mg=\mg_- \oplus \mg_0 \oplus \mg_+$ is called \textbf{Jordan graded} if 
\[
[\mg_+,\mg_-]=\mg_0 \;\text{ and }\; \mg_0 \cap Z(\mg)=0.
\]
\end{definition}

We have the following result by Lemmata $4$ and $5$ in \cite{Krutelevich}. 
\begin{proposition}\label{TKK inverse}
For every Jordan superpair $(V^{\plus},V^{\minus})$, we have
\[
\mathcal{J} (\TKK(V^{\plus},V^{\minus}))\cong (V^{\plus},V^{\minus}).\]
Let $\mg$ be a Jordan graded Lie superalgebra, then $ \TKK(\mathcal{J}(\mg)) \cong \mg.$
\end{proposition}
Note that the main results in \cite{Krutelevich} are only concerned with finite dimensional pairs, over algebraically closed fields with characteristic zero. However, the mentioned lemmata still hold for arbitrary Jordan superpairs over a field with characteristic different from $2$ or $3$.

\subsection{Connection}
The main result of this section is the following proposition, which shows that Kantor's and Koecher's constructions for unital Jordan superalgebras coincide.

\begin{proposition}\label{PropKan}
For a unital Jordan superalgebra $V$, we have $\Kan(V) \cong \TKK(V)$.
\end{proposition}
\begin{proof}
 Let $\Kan(V) = \mg_+ \oplus \mg_0 \oplus \mg_-$. The relations $[P,a]=L_a$ and
 $[[L_a,P],b]=[L_a,L_b]-L_{a  b}$ in Subsection~\ref{section TKK for Jordan superalgebras} imply $[\mg_-,\mg_+]=\mg_0.$ For all $x\in \mg_0$, it follows from the definition of the bracket that, if $[x,\mg_-]=0 $, then $x=0$. So $\mg_0 \cap Z(\mg)=0$ and $\Kan(V)$ is Jordan graded. Hence Proposition~\ref{TKK inverse} implies $$\TKK(\mathcal{J}(\Kan(V))) \cong \Kan(V).$$

Set $(V^{\plus}, V^{\minus}):=\mathcal{J}(\Kan(V))$. Then $V^{\minus} = V$ and $V^{\plus} = \langle P, [L_a,P] \mid a \in V\rangle.$
One can check that the map $\phi\colon (V^{\plus}, V^{\minus})\to (V,V)$ defined by 
\begin{itemize}
\item $\phi(x)= x $ for $x \in  V^{\minus}$,
\item $\phi(P) = -\frac{e}{2},$ where $e$ is the unit of $V$,
\item $\phi([L_a,P]) = \frac{a}{2}$.
\end{itemize}
is an isomorphism of Jordan pairs.
From this it follows that $$\TKK(V)=\TKK(V,V) \cong \TKK(\mathcal{J}(\Kan(V))) \cong \Kan(V),$$ which proves the proposition.
\end{proof}

\begin{remark}{\rm
The proposition as stated does not extend to Jordan superalgebras without multiplicative identity $e$. If $V$ is finite dimensional but not unital, we will generally have
$$\dim \Kan(V)_{\plus} \;\not=\; \dim V=\dim\TKK(V)_{\plus},$$
and hence $\Kan(V)\not\cong \TKK(V)$. This difference in dimension can for instance be caused by the occurrence of elements of $V$ for which the left multiplication operator is trivial , since this lowers the dimension of $\langle P, [L_a,P] \rangle$, or by $P\not\in\langle [L_a,P] \rangle$, which raises the dimension.

Another source of counterexamples comes from Jordan superalgebras $V$ which satisfy $ \Inn(V,V)\not=\istr(V)$, see {\it e.g.} Remark \ref{RemInn}.}
\end{remark}

\subsection{Tits' approach.} \label{Section Tits' approach}
There is a third version of the TKK construction, which appeared in \cite{Tits} and historically was the first to appear. 
In this section, we will give the super version of this construction by Tits. 

Consider an arbitrary Jordan superalgebra $V$. Let $\cD$ be a Lie superalgebra, containing $\Inn(V)$,
with a Lie superalgebra morphism
\begin{align*}
\psi:\;\cD\,\to\, \Der(V);\qquad d\mapsto \psi_d,
\end{align*}
such that $\psi$ acts as the identity on the subalgebra $\Inn(V)$.
Finally, let $Y$ be an arbitrary three-dimensional simple Lie algebra $Y$. For example, for $\mK=\mC$, we only have $Y\cong\mathfrak{sl}_{2}(\mC)$ and for $\mK=\mR$ either $Y\cong \mathfrak{sl}_{2}(\mR)\cong\mathfrak{su}(1,1)$ or $Y\cong\mathfrak{su}(2)$. Let $(y,y'):= \frac{1}{2} \tr(\ad(y)\ad(y'))$ be the Killing form on $Y$.

Then we define a Lie superalgebra
$$\TT(V,\cD,Y)\;:=\; \cD\,\oplus \, (Y\otimes V),$$
where $\cD$ is a subalgebra, and the rest of the multiplication is defined by
\begin{align*}
[d, y\otimes v]&= y \otimes \psi_d(v), \\
[y\otimes v , y'\otimes v']&= (y,y') [L_v,L_{v'}] + [y,y']\otimes vv',
\end{align*}
for arbitrary $d\in \cD$, $y,y'\in Y$ and $v,v'\in V$.  For $Y=\mathfrak{sl}_2(\mK)$ we can use the $3$-grading on $\mathfrak{sl}_2(\mK)$ to define a $3$-grading on $\TT(V,\cD,\mathfrak{sl}_2(\mK))$: 
\[
\TT(V,\cD,\mathfrak{sl}_2(\mK))_{\minus } = \mathfrak{sl}_2(\mK)_{\minus} \otimes V,\quad \TT(V,\cD,Y)_{0 } = \cD \oplus (\mathfrak{sl}_2(\mK)_{0} \otimes V), \quad \TT(V,\cD,\mathfrak{sl}_2(\mK))_{\plus } = \mathfrak{sl}_2(\mK)_{\plus} \otimes V.
\]

For unital Jordan superalgebras, this contains, as a special case, Koecher's and hence also Kantor's construction, as we prove in the following proposition.
\begin{proposition}\label{proposition Koecher and Tits are isomorphic}
For a unital Jordan superalgebra $V$ we have
$$\TT\left(V,\Inn(V),\mathfrak{sl}_{2}(\mK)\right)\;\cong\; \TKK(V).$$
\end{proposition}
\begin{proof}
Consider a $\mK$-basis $e,f,h$ of $\mathfrak{sl}_{2}(\mK)$, such that $[e,f]=h$, $[h,e]=2e$ and $[h,f]=-2f$.
For a unital Jordan superalgebra $V$, we have $\Inn(V,V)\cong \istr(V)$, by Proposition~\ref{inner structure is isomorphic to inner derivations}(2).

Then an isomorphism between $\TT\left(V,\Inn(V),\mathfrak{sl}_{2}(\mK)\right)$ and $\TKK(V)= V^{\plus} \oplus \istr(V) \oplus V^{\minus}$ is given by
\begin{align*}
e\otimes a \mapsto a^{\plus},\quad f\otimes a \mapsto a^{\minus}, \quad h\otimes a \mapsto 2L_a, \quad [L_a,L_b] \mapsto [L_a,L_b].
\end{align*}
It follows from the definitions that this is a Lie superalgebra morphism.
\end{proof}
\begin{remark}\label{RemTits} {\rm
From the proof, it is clear that the proposition still holds for non-unital Jordan superalgebras so long as $\istr(V)\cong \Inn(V,V)$.}
\end{remark}

Now we consider the opposite direction of the above construction. Let $N$ be a Lie superalgebra and $Y$ a simple Lie algebra of dimension $3$. We say that $Y$ acts on $N$ if there is an (even) Lie superalgebra homomorphism from $Y$ to $\Der(N)$. For example, we can define an action of $Y$ on $\TT\left(V,\cD,Y\right)$ as follows
\[
y\cdot ( d + y' \otimes v)= [y,y'] \otimes v. 
\]
Under this action, $\TT\left(V,\cD,Y\right)$ viewed as an $Y$-module decomposes as a trivial part given by $\cD$ and $\dim(V)$ copies of the adjoint representation. 

Now consider an arbitrary Lie superalgebra $N$ with $Y$-action which decomposes as above, {\it viz.} as a trivial representation $\cD$ and some copies of the adjoint representation,
\[
N= \cD \oplus (Y\otimes A),
\]
for some vector space $A$. As a direct generalisation of \cite{Tits}, we show
that there is a Jordan algebra structure on $A$ where $\cD$ acts on $A$ by derivations and $\TT\left(V,\cD,Y\right)$ is the inverse of this construction.

\begin{proposition}
Let $N$ be a Lie superalgebra and $Y$ a $3$-dimensional simple Lie algebra which acts on $N$ such that $N$ decomposes as $N=\cD\oplus (Y\otimes A)$ where $\cD$ is a trivial representation and $Y$ the adjoint representation. 
Then $A$ is a Jordan superalgebra and $\cD$ is a superalgebra containing the inner derivations on $A$ equipped with a morphism $\psi\colon \cD \to \Der(A)$, for which the restriction to the inner derivations is the identity. Furthermore 
\[ N\cong \TT(A,\cD,Y).\]
\end{proposition}
\begin{proof}
Proposition 1 in \cite{Tits} and its proof, which extend trivially to the super case, imply that under these conditions, $\cD$ is a subalgebra of $N$, and there are bilinear maps $\alpha(\cdot,\cdot):\cD\times A\to A$, $\langle\cdot,\cdot\rangle: A\times A\to \cD$ and $\mu: A\times A\to A$, such that
$$[d,y\otimes a]=y\otimes \alpha(d,a)\quad\mbox{ and }\quad [y\otimes a,y'\otimes a']=(y,y')\langle a,a'\rangle +[y,y']\mu(a,a').$$
Furthermore $(A,\mu)$ is a Jordan superalgebra and $d\mapsto\alpha(d,\cdot)$ is a Lie superalgebra morphism $\phi:\cD\to \Der(A)$. Finally, by equation~(2.6) in \cite{Tits}, we have
$$\phi(\langle a,b\rangle) =[L_a,L_b].$$
Comparison with the definition of $\TT(A,\cD,Y)$ concludes the proof.
\end{proof}

\section{Further TKK constructions}\label{SecKoe}

In this section we consider variations of the TKK constructions for a Jordan superalgebra $V$, which also appear in the literature, by using $\str(V )$ and $\Der(V,V)$,
instead of  
$\istr(V )$ and $\Inn(V,V)$.

\subsection{Definition} \label{section definition further TKK}
The Lie superalgebra $\TT(V,\cD,\mathfrak{sl}_2)$ had more freedom compared to the constructions by Kantor and Koecher, due to the choice of $\cD$. Also in the Koecher construction, we can replace $\Inn(V^{\plus},V^{\minus})$ by any Lie superalgebra containing $\Inn(V^{\plus},V^{\minus})$ with a morphism to $\Der(V^{\plus},V^{\minus})$ which restricts to the identity on the inner derivations. For example, we can set $\mg_{0}=\Der(V^{\plus},V^{\minus})$ in the TKK construction of Section~\ref{section TKK for Jordan superpairs}. This gives a $3$-graded Lie superalgebra 
\[\widetilde{\TKK}(V^{\plus},V^{\minus})=V^{\plus} \oplus \Der(V^{\plus},V^{\minus}) \oplus V^{\minus}  ,\] see  \cite{Garcia} for more details. Remark that, by construction, $\TKK(V^{\plus},V^{\minus})$ is an ideal in $\widetilde{\TKK}(V^{\plus},V^{\minus})$. We will again use the notation $\widetilde{\TKK}(V)$ for $\widetilde{\TKK}(V,V)$. In Subsection~\ref{SubOut}, we will prove that $\widetilde{\TKK}(V)$ is the superalgebra of derivations of $\TKK(V)$ for unital Jordan superalgebras. 
We can also relate $\widetilde{\TKK}(V)$ to the Tits' construction in Subsection~\ref{Section Tits' approach} as follows.
\begin{lemma} \label{Lemma connection Tits generalized Koecher} For a unital Jordan superalgebra $V$, we have \[ \widetilde{\TKK}(V)\cong \TT(V,\Der(V),\mathfrak{sl}_{2}(\mK)).\]
\end{lemma}
This will be proved in greater generality in Subsection~\ref{SubT}.

\subsection{Comparison of further TKK constructions}\label{SubT}

Let $\cD$ be a Lie superalgebra containing $\Inn(V)$ with a morphism $\psi$ to $\Der(V)$ such that $\psi_{|\Inn(V)}=\id$. Define the Lie superalgebra  \[ \widetilde{\cD} :=\cD \oplus \{ L_x \mid x \in V\},
\]
where $\cD$ is a subalgebra of $\widetilde{\cD}$, the product of $L_x$ and $L_y$ is given by $[L_x,L_y]$ interpreted via the embedding of $\Inn(V)$ in $\cD$, and \[
[D,L_x]:=L_{\psi(D)x}\qquad \mbox{for $D\in\cD$, $x\in V$}.
\]
Set \[ \widetilde{\psi}\colon \widetilde{\cD} \to \Der(V,V); \quad D +L_x \mapsto (\psi(D)+L_x, \psi(D)-L_x). \]
From Lemma \ref{Lemma: left multiplication and D are derivations}, it follows that this map is well defined, while from the definition of the bracket on $\widetilde{\cD}$ it follows that it is a Lie superalgebra morphism.
The morphism $\widetilde{\psi}$ yields an action of $\widetilde{\cD}$ on $V^{\plus}$ and $V^{\minus}$, which allows us to define a ${\rm TKK}$ construction similar to the Koecher construction in Subsection~\ref{section TKK for Jordan superpairs}. Concretely, the bracket on 
$$\TKK_{\cD}(V):=V\oplus \widetilde{\cD}\oplus V$$ is given by
\begin{align*}
[x,u]=2L_{xu}+2[L_x,L_u], \quad 
[d,x]=\widetilde{\psi}(d)x, \quad [d,u]=\widetilde{\psi}(d)u, \\ [d_1,d_2]=[d_1,d_2]_{\widetilde{\cD}}, \quad [x,y]=0=[u,v],
\end{align*} for $x,y$ in $V^{\plus}$, $u,v$ in $V^{\minus}$,  $d,d_1,d_2$ in $\widetilde{D}$ and $[\cdot,\cdot]_{\widetilde{\cD}}$ the product in $\widetilde{\cD}$.
\begin{proposition}\label{Propnu} Consider a (not necessarily unital) Jordan superalgebra $V$ and a Lie superalgebra $\cD$ as above. We have an isomorphism of Lie superalgebras  \[ \TT(V, \cD, \mathfrak{sl}_{2}(\mK))\; \cong\; \TKK_{\cD}(V).\]
\end{proposition}
\begin{proof}
The following generalisation of the map used in Proposition~\ref{proposition Koecher and Tits are isomorphic} 
 \begin{align*}
e\otimes a \mapsto a^{\plus},\quad f\otimes a \mapsto a^{\minus}, \quad h\otimes a \mapsto 2L_a, \quad D \mapsto D,
\end{align*}
is an isomorphism between $\TT(V, \cD, \mathfrak{sl}_{2}(\mK))$ and $\TKK_{\cD}(V)$.
\end{proof}
The case $\cD=\Inn(V)$ yields \[ \TT(V,\Inn(V),\mathfrak{sl}_{2}(\mK))\cong \TKK_{\Inn(V)}(V),\]
with $\widetilde{\Inn(V)} =\{ L_x\mid x \in V \} \oplus \Inn(V)$. This is a generalisation of Proposition~\ref{proposition Koecher and Tits are isomorphic} to the non-unital case.

Note that if there exists an $x\in V$, such that $L_x$ is in $\Der(V)$, then  $\widetilde{\Der(V)}$ contains two copies of $L_x$, one in $\Der(V)$ and one in $\{L_x \mid x \in V \}.$ The isomorphism between $\widetilde{\Der(V)}$ and $\Der(V,V)$ maps the first to $(L_x,L_x)$, while the second copy gets mapped to $(L_x,-L_x)$.
For unital Jordan superalgebras we have canonical isomorphisms $\widetilde{\Inn(V)} \cong \istr(V)$ and $\widetilde{\Der(V)}\cong \str(V)$, by Remark \ref{Remark decomposition structure algebra for unital Jordan algebras}, and thus  $\TKK_{\Inn(V)}(V) = \TKK(V)$ and $\TKK_{\Der(V)}(V)= \widetilde{\TKK} (V).$ Hence we find that Proposition~\ref{Propnu} implies Lemma  \ref{Lemma connection Tits generalized Koecher}.

\begin{remark}{\rm Let $\mg$ be an arbitrary $3$-graded Lie superalgebra and set  $(V^{\plus},V^{\minus})= \mathcal{J}(\mg)$. Then we have a morphism of Lie superalgebras  \[ \mg_{0} \to \Der(V^{\plus},V^{\minus}); \; x \mapsto (\ad_x\mid_{\mg_{\plus}}, \ad_x\mid_{\mg_{\minus}}), \]
 and  its kernel $I$
is an ideal in $\mg_{0}$ and by construction even in $\mg$. By definition of $\widetilde{\TKK}(V^{\plus},V^{\minus})$, we have an embedding of $\mg / I $ into $\widetilde{\TKK}(V^{\plus},V^{\minus})$.
If $\mg= \TT\left(V,\cD,\mathfrak{sl}_{2} \right)$ for a unital Jordan superalgebra $V$, then one can easily check that $I=0$ (and thus $\cD \subseteq \Der(V)$) is equivalent with the condition that the only ideal of $\mg$ contained in $\cD$ is the zero ideal.}
\end{remark}

Another ``universality property'' of $\widetilde{\TKK}(V^+,V^-)$ will be discussed in Subsection~\ref{SecTom}.
\subsection{Outer derivations}\label{SubOut}
\begin{definition}[See \cite{Extensions}]
For a Lie superalgebra $\mg$, denote the Lie superalgebra of derivations by $\Der(\mg)$. The inner derivations $\Inn(\mg)=\{\ad_X\,|\, X\in\mg\}$ form an ideal isomorphic to the quotient of $\mg$ by its centre. The Lie superalgebra of outer derivations is $\Out (\mg) = \Der(\mg) / \Inn( \mg)$.

An extension $\mathfrak{e}$ of a Lie superalgebra $\fg$ over a Lie superalgebra $\fh$ is a Lie superalgebra $\mathfrak{e}$ such that the following is a short exact sequence:
\[
0\to \fh \to \mathfrak{e} \to \fg \to 0.
\]
In particular $\fh$ is an ideal in $\mathfrak{e}$. 
\end{definition}
Let $\fh$ be a Lie superalgebra with trivial centre. Then we will freely use the isomorphism between the space of extensions of $\fg$ over $\fh$, and the space of Lie superalgebra morphisms $\mathfrak{g}\to \Out(\fh)$, see e.g. Corollary~8 in \cite{Extensions}.

The main result of this section is the following proposition.
\begin{proposition} \label{Theorem of outer derivations} For a unital Jordan superalgebra $V$, we have \[ \widetilde{\TKK}(V) \cong \Der(\TKK(V)),\] and thus
\[ \widetilde{\TKK}(V) / \TKK(V) \cong  \str(V) / \istr(V) \cong \Out (\TKK(V) ).\]
\end{proposition}

\begin{remark}{\rm
Again the assumption of a multiplicative identity is essential for this proposition. A counterexample of the statement for non-unital Jordan superalgebras is given in Subsection~\ref{sec: example Jordan superalgebra K}.}
\end{remark}

\begin{remark}\label{RemGrad}
{\rm For any $\mZ$-graded Lie superalgebra, the Lie superalgebra $\Der(\fg)\subset \End_{\mK}(\fg)$ is $\mZ$-graded by construction. The endomorphisms in $\Der(\fg)_i$ map elements in $\fg_j$ to elements in~$\fg_{i+j}$. Clearly $\Inn(\fg)$ is then a graded ideal in $\Der(\fg)$, so that $\Out(\fg)$ is also $\mZ$-graded. In particular, when $\mg$ is $3$-graded then $\Der(\mg)$ and $\Out(\mg)$ will be $5$-graded.}
\end{remark}

The following reformulation of Proposition~\ref{Theorem of outer derivations} holds for arbitrary Jordan superpairs and thus {\it a fortiori} also for non-unital Jordan superalgebras. \begin{proposition}\label{proposition Jordan superpairs and outer derivations}
For a Jordan superpair $(V^{\plus},V^{\minus})$, we have
\[
\widetilde{\TKK}(V^{\plus},V^{\minus}) /  \TKK(V^{\plus},V^{\minus})  \cong \Der(V^{\plus},V^{\minus}) / \Inn(V^{\plus},V^{\minus}) \cong \Out( \TKK(V^{\plus},V^{\minus}))_{0}.
\]
In particular, for a (non-unital) Jordan superalgebra $V$ we have that $\widetilde{\TKK}(V)$ is the extension over $\TKK(V)$ of $\Out(\TKK(V))_0$ corresponding to the embedding $\Out(\TKK(V))_0\hookrightarrow \Out(\TKK(V)).$
\end{proposition} 

The rest of the subsection is devoted to the proofs of Propositions \ref{Theorem of outer derivations} and \ref{proposition Jordan superpairs and outer derivations}.


\begin{lemma} \label{Der(V,V) = Der0}
We have a Lie superalgebra isomorphism
\[ \phi: \Der(V^{\plus},V^{\minus}) \;\tilde{\to}\;  \Der(\TKK(V^{\plus},V^{\minus}))_{0},\;\; x\mapsto \ad_x|_{\TKK(V^{\plus},V^{\minus})}.
\]
\end{lemma}
\begin{proof}
As  $\TKK(V^{\plus},V^{\minus})$ is an ideal in $\widetilde{\TKK}(V^{\plus},V^{\minus})$ and $\Der(V^{\plus},V^{\minus})\subset \widetilde{\TKK}(V^{\plus},V^{\minus})$ is the zero component of the $\mZ$-grading, the map $\phi$ is well-defined. By construction it is an injective Lie superalgebra morphism.

Now let $D$ be a $\mZ$-grading preserving derivation of $\TKK(V^{\plus},V^{\minus})$, then $(D|_{V^{\plus}},D|_{V^{\minus}})$ is an element of $\Der(V^{\plus},V^{\minus})$ since, using the definition of the bracket on $\TKK(V^{\plus},V^{\minus})$ in Subsection~\ref{section TKK for Jordan superpairs}, we find
\begin{align*}
D(\{x,y,z\}) = D([[x,y],z]) = [[D(x),y],z]+ (-1)^{\abs{x}\abs{D}}[[x,D(y)],z]+ (-1)^{(\abs{x}+\abs{y})\abs{D}} [[x,y],D(z)] \\
= \{D(x),y,z\}+ (-1)^{\abs{x}\abs{D}}\{x,D(y),z\}+ (-1)^{(\abs{x}+\abs{y})\abs{D}} \{x,y,D(z)\}.
\end{align*} One can check that  \[D= \ad_{(D|_{V^{\plus}},D|_{V^{\minus}})} \text{ and } x=(\ad_x|_{V^{\plus}},\ad_x|_{V^{\minus}}). \] So we have indeed $\Der(V^{\plus},V^{\minus}) \cong\Der(\TKK(V^{\plus},V^{\minus}))_{0}$
as Lie superalgebras.
\end{proof}
Using this lemma, we can immediately prove Proposition~\ref{proposition Jordan superpairs and outer derivations}.
\begin{proof}[Proof of Proposition~\ref{proposition Jordan superpairs and outer derivations}]
The first isomorphism follows immediately from Koecher's construction in Subsection~\ref{section TKK for Jordan superpairs}. Furthermore, since the intersection of the centre of $\TKK(V^{\plus},V^{\minus})$ with $\TKK(V^{\plus},V^{\minus})_0$ is trivial, we have $$\Inn(\TKK(V^{\plus},V^{\minus}))_0\cong \TKK(V^{\plus},V^{\minus})_0 =\Inn(V^{\plus},V^{\minus}).$$ Hence, from Lemma \ref{Der(V,V) = Der0} we conclude that $\Der(V^{\plus},V^{\minus})/ \Inn (V^{\plus},V^{\minus}) \cong \Out(  \TKK(V^{\plus},V^{\minus}))_{0}$.
\end{proof}

To prove Proposition~\ref{Theorem of outer derivations}, we will show that, for a unital Jordan superalgebra $V$, all outer derivations of $\TKK(V)$ are grading preserving for the $3$-grading we consider. This is not true for non-unital algebras, see Subsection~\ref{sec: example Jordan superalgebra K}.
\begin{lemma} \label{Der2 = 0}
For a unital Jordan superalgebra $V$, we have
\[ \Der(\TKK(V))_{-2}  =0 = \Der(\TKK(V))_{2}.
\] 
\end{lemma}
\begin{proof}
Let $D \in \Der(\TKK(V))_{2}$. First remark that $D$ acts trivially on $\TKK(V)_0$ and $\TKK(V)_1$. We will show that it must also acts trivially on $\TKK(V)_{\minus 1}$. To use the definition of $\TKK(V)$ we use the Jordan superpair $(V^{\plus},V^{\minus}):=(V,V)$.
For $x\in V$, we use the notation $x^{\plus}$ and $x^{\minus}$, as in Example~\ref{ExJD}.
We find, using the definition of the bracket on $\TKK(V)$ and the property $D(\TKK(V)_0)=0$, that
\[ D(e^{\minus})= \frac{1}{2}D([e^{\minus}, \mD_{e,e}])=\frac{1}{2}[D(e^{\minus}), \mD_{e,e}]+\frac{1}{2}[e^{\minus}, D(\mD_{e,e})]=\frac{1}{2}[D(e^{\minus}), \mD_{e,e}]= -D(e^{\minus}).\]
Hence $D(e^{\minus})=0$, which then implies that
$$D(x^{\minus})=\frac{1}{2} D([e^{\minus}, \mD_{x,e} ])=\frac{1}{2}[D(e^{\minus}), \mD_{x,e}]+\frac{1}{2}[e^{\minus}, D(\mD_{x,e})]= 0.$$ We conclude that $D=0$ for all $D  \in \Der(\TKK(V))_{2}$.
The proof that $\Der(\TKK(V))_{-2}=0$ is completely similar.
\end{proof}
\begin{lemma} \label{V cong Der1}
For a unital Jordan superalgebra $V$, we have isomorphisms
\begin{align*}
V\;\tilde{\to}\; \Der(\TKK(V))_{1}; \quad x \mapsto \ad_{x^{\plus}} \text{ and }\quad V\; \tilde{\to}\; \Der(\TKK(V))_{\minus 1}; \quad x \mapsto \ad_{x^{\minus}},
\end{align*}
as super vector spaces. 
\end{lemma}
\begin{proof}
Let $x^{\plus}$ be an element of $\TKK(V)_1=V^{\plus}$, then $\ad_{x^{\plus}}\in \Der(\TKK(V))_{1}$. With an element $D$ in $\Der(\TKK(V))_{1}$ we can associate the element $-\frac{1}{2} D(\mD_{e,e}) \in V^+$.  
We will now show that $x^+ \mapsto \ad_{x^{\plus}}$ and $D \mapsto -\frac{1}{2} D(\mD_{e,e})$ are each others inverse. This follows from
\[
-\frac{1}{2}\ad_{x^{\plus}}\left( \mD_{e,e}\right)= -\frac{1}{2}[x^{\plus}, \mD_{e,e}]=x^+
\]
and the following three calculations, for arbitrary $x,y\in V$,
\begin{align*}
-\frac{1}{2}\ad_{D(\mD_{e,e})}(y^{\minus})& = -\frac{1}{2}[D(\mD_{e,e}), y^{\minus}]=-\frac{1}{2}D([\mD_{e,e},y^{\minus}]) +\frac{1}{2}[\mD_{e,e},D(y^{\minus})] = D(y^{\minus}) \\
-\frac{1}{2}\ad_{D(\mD_{e,e}) }(L_x) &= -\frac{1}{2}[D(\mD_{e,e}), L_x]=-\frac{1}{2}D([\mD_{e,e},L_x]) +\frac{1}{2}[\mD_{e,e},D(L_x)] = D(L_x) \\
-\frac{1}{2}\ad_{D(\mD_{e,e}) }([L_x,L_y]) &= -\frac{1}{2}[D(\mD_{e,e}), [L_x,L_y]]=-\frac{1}{2}D([\mD_{e,e},[L_x,L_y]])  +\frac{1}{2}[\mD_{e,e},D([L_x,L_y])]\\ &= D([L_x,L_y]). 
\end{align*}
We conclude that $V\cong \Der(\TKK(V))_{1}$. Similarly $ \TKK(V)_{\minus 1}  \to \Der(\TKK(V))_{\minus 1}; \; x^- \mapsto \ad_{x^{\minus}} $ is an isomorphism with inverse $D \mapsto \frac{1}{2}D(\mD_{e,e})$.
\end{proof}

\begin{proof}[Proof of Proposition~\ref{Theorem of outer derivations}]
Consider the following morphism of Lie superalgebras
\begin{align*}
\widetilde{\TKK}(V)\to \Der(\TKK(V)); \;x \mapsto {\ad_x}|_{\TKK(V)}.	
\end{align*}
Combining Lemmata \ref{Der(V,V) = Der0}, \ref{Der2 = 0} and \ref{V cong Der1}, we see that this is an isomorphism. 
\end{proof}

\subsection{Alternative construction}\label{SecTom}

The construction of $\widetilde{\TKK}(V^{\plus},V^{\minus})$ starting from $\TKK(V^{\plus},V^{\minus})$ in Proposition~\ref{proposition Jordan superpairs and outer derivations} fits into a more general construction.
In \cite[Section 4.1]{Tom}, the authors start from an arbitrary $(2n+1)$-graded Lie superalgebra $\cL= \bigoplus_{i\in \mZ} \cL_i$ (strictly speaking only Lie algebras are considered, but the procedure carries over naturally to the super case). Then \cite[Construction~4.1.2]{Tom} constructs an extension $\overline{\cL}$ over $\cL$, which is again a $(2n+1)$-graded Lie superalgebra which satisfies $\overline{\cL}_i=\cL_i$ if $i\not=0$.

It is not difficult to show that in the case of a $3$-graded Lie superalgebra $\cL$ we have $\overline{\cL_0} = \Der(\cL^{\plus}, \cL^{\minus})$ and hence 
$$\overline{\cL} = \widetilde{\TKK} (\cL^{\plus},\cL^{\minus})\qquad\mbox{with}\quad(\cL^{\plus}, \cL^{\minus}):= \cJ (\cL),$$ the Jordan superpair associated with $\cL$ in Subsection~\ref{section TKK for Jordan superpairs}. In other words,
$$\widetilde{\TKK}(V^{\plus},V^{\minus})\;\cong\;\overline{\TKK(V^{\plus},V^{\minus})}.$$
This reveals a universality principle behind $\widetilde{\TKK}(V^{\plus},V^{\minus})$, as the construction of $\overline{\cL}$ starting from~$\cL$ in~\cite{Tom} does not depend on $\cL_0$.

An interesting consequence of \cite[Lemma~4.1.3]{Tom} is then
$$\Out(\widetilde{\TKK}(V^{\plus},V^{\minus}))\;=\;0,$$
for arbitrary Jordan superpairs $(V^{\plus},V^{\minus})$, so also for arbitrary (unital or non-unital) Jordan superalgebras.

\section{Examples}\label{SecEx}
In this section, we use the results of the previous sections to calculate $\widetilde{\TKK}(V)$ for $V$ any finite dimensional simple Jordan superalgebra over an \emph{ algebraically closed field of characteristic zero.} We assume these conditions on the ground field for the entire section.

\subsection{Unital finite dimensional simple  Jordan superalgebras.} \label{section: Examples of TKK constructions}
A complete list of unital finite dimensional simple  Jordan superalgebras $V$ and the corresponding $\Kan(V)$ is given in \cite{Kac, Cantarini}. This gives us $\TKK(V)$ and $\TT(V,\Inn(V),\mathfrak{sl}_{2})$, by Propositions~\ref{PropKan} and \ref{proposition Koecher and Tits are isomorphic}. For the Jordan superalgebras we use the notation of \cite{Cantarini}, where also the definitions can be found. We introduce our convention for the notation of Lie superalgebras in Appendix A. 
In  \cite[Theorem 5.1.2]{Kac Lie superalgebras} and \cite[Chapter~III, Proposition~3]{Scheunert}, $\Der (\mg)$ is calculated for any simple finite dimensional Lie superalgebra  $\mg$.
Together with Proposition~\ref{Theorem of outer derivations} and Lemma~\ref{Lemma connection Tits generalized Koecher}, this gives us $\widetilde{\TKK}(V)\cong \TT(V,\Der(V),\mathfrak{sl}_{2})$, leading to the following table.

\begin{center}
\begin{tabular}{|c|c|c|c|}
\hline
 $V$ &  $\TKK(V)$ & $\widetilde{\TKK}(V)$ & Remarks\\ 
\hline\hline 
$gl(m,n)_{\plus}$ &  $\mathfrak{sl}(2m|2n)$ &   &$m\not=n$\\ 
\hline 
$gl(m,m)_{\plus}$ & $\mathfrak{psl}(2m|2m)$ & $\mathfrak{pgl}(2m|2m)$ & $m>1$\\ 
\hline
$osp(m, 2n)_{\plus}$& $\mathfrak{osp}(4n|2m)$  &  & $(n,m)\not=(1,0 )$\\ 
\hline 
$(m-3,2n)_{\plus}$ & $\mathfrak{osp}(m|2n)$ & & $m\geq 3$, $(m,2n) \not= (4,0)$ \\ 
\hline 
$p(n)_{\plus}$ & $\mathfrak{spe}(2n)$ &  $\mathfrak{pe}(2n)$ & $n>1$\\ 
\hline 
$q(n)_{\plus}$ & $\mathfrak{psq}(2n)$ & $\mathfrak{pq}(2n)$ & $n>1$\\ 
\hline 
$D_t$  & $D(2,1,t)$  & & $t\not\in \{0, -1\} $\\ 
\hline 
$E$ &  $E(7)$  & &\\ 
\hline 
$F$ &  $F(4)$ &  &\\ 
\hline 
$JP(0,n-3)$   & $H(0,n)=H(n)$ &  $ \mK C \ltimes \widetilde{H}(n)$ & $n\geq 5$\\ 
\hline 
$gl(1,1)_{\plus}$& $\mathfrak{psl}(2|2)$ & $D(2,1,-1)$ &\\ 
\hline 
\end{tabular} 
\end{center}
When $\widetilde{\TKK}(V)$ is isomorphic to $\TKK(V)$, we only wrote it once.

Taking the zero component of the 3-graded algebras in the above table gives us $\istr(V)\cong \Inn(V,V)$ and $\str(V)\cong \Der(V,V)$. These are listed in the following table, where the same restrictions on the indices are assumed as in the previous table.
\begin{center}

\begin{tabular}{|c|c|c|}
\hline
 $V$ & $\istr(V)$ & $\str(V)$ \\ 
\hline\hline 
$gl(m,n)_{\plus}$ & $\mathfrak{sl}(m|n)\oplus \mathfrak{sl}(m|n) \oplus \mK$ &\\ 
\hline 
$gl(m,m)_{\plus}$ &$s(\mathfrak{gl}(m|m)\oplus \mathfrak{gl}(m|m) )/ \langle I_{4m} \rangle $ & $(\mathfrak{gl}(m|m)\oplus \mathfrak{gl}(m|m))/ \langle I_{4m} \rangle$\\ 
\hline
$osp(m, 2n)_{\plus}$& $\mathfrak{gl}(2n|m)$ &    \\ 
\hline 
$(m-3,2n)_{\plus}$ &$\mathfrak{osp}(m-2|2n)\oplus \mK$ & \\ 
\hline 
$p(n)_{\plus}$ & $\mathfrak{sl}(n|n) $ &$\mathfrak{gl}(n|n)$ \\ 
\hline 
$q(n)_{\plus}$ &$s(\mathfrak{q}(n) \oplus \mathfrak{q}(n))/\langle I_{4n}\rangle $ & $(\mathfrak{q}(n) \oplus \mathfrak{q}(n))/\langle I_{4n} \rangle  $\\ 
\hline 
$D_t$ &$\mathfrak{sl}(2|1)\oplus \mK \cong \mathfrak{osp}(2|2)\oplus  \mK$ & \\ 
\hline 
$E$ & $E(6) \oplus \mK$& \\ 
\hline 
$F$ & $\mathfrak{osp}(2|4)\oplus \mK$ &   \\ 
\hline 
$JP(0,n-3)$  &$\widetilde{H}(n-2) \ltimes  \left( \Lambda(n-2) / \langle \xi_1\cdots \xi_{n-2} \rangle  \right)$ &$\mK C \ltimes \left(\widetilde{H}(n-2) \ltimes  \Lambda(n-2)\right)  $ \\ 
\hline 
$gl(1,1)_{\plus}$ & $s(\mathfrak{gl}(1|1) \oplus \mathfrak{gl}(1|1)) / \langle I_4 \rangle $  & $\mathfrak{sl}_{2} \ltimes \istr(gl(1,1)_{\plus}) $ \\ 
\hline 
\end{tabular} 
\end{center}
 Again, if $\str(V)$ is isomorphic to $\istr(V)$, we only wrote it once and the notation is explained in Appendix~A.
 The action of $\mathfrak{sl}_{2} $ on $\istr( gl(1,1)_{\plus})$  is the adjoint action by using the embedding of $\mathfrak{sl}_{2} $ in $D(2,1;-1)$.
The following isomorphisms exist in the list of Jordan superalgebras:
\[
(1,2)_{\plus} \cong D_{1}, \quad D_t \cong D_{t^{-1}}.
\]
Furthermore, also the simple Jordan superalgebras $JP(0,1)$ and $D_{-1}$ appear in the literature, but they are isomorphic to $gl(1,1)_{\plus}$, so they are already included in the table.

\subsection{The non-unital finite dimensional simple Jordan superalgebra.}
\label{sec: example Jordan superalgebra K}
The full list of finite dimensional simple Jordan superalgebras in \cite{Kac, Cantarini} contains only one Jordan superalgebra which is non-unital. In~\cite{Kac} it is denoted by $K$. 
The algebra $K$ is defined as
\[
K = \langle a \rangle \oplus \langle \xi_1, \xi_2 \rangle, \quad \abs{a}=\oa, \; \abs{\xi_1}=\abs{\xi_2}= \ob,
\]
with multiplication satisfying $a^2=a$, $a\xi_i= \frac{1}{2}\xi_i$ and $\xi_1\xi_2=a.$

A straightforward calculation implies \[ \istr (K)=\str(K)= \Inn(K,K)  \cong \mathfrak{sl}(1|2), \text{ and } \Der(K,K) \cong \mathfrak{gl}(1|2). \]
This gives a counterexample to the statement in Proposition~\ref{inner structure is isomorphic to inner derivations}(1) for non-unital Jordan superalgebras. For $K$, the sums in Definitions~\ref{Definition structure algebra} and~\ref{Definition inner structure algebra} are direct.

One also finds
 $$\TKK(K)\cong \mathfrak{psl}(2|2).$$
 By construction, $\widetilde{\TKK}(K)$ is an extension over $\TKK(K)$. As $\istr(K)\cong \Inn(K,K)$, it follows easily that the same is true for $\Kan(K)$. The algebras $\widetilde{\TKK}(K)$ and $\Kan(K)$ can hence be described in terms of $\Out(\TKK(K))\cong \mathfrak{sl}_{2}$:\begin{itemize}
\item $\widetilde\TKK(K)\cong \mathfrak{pgl}(2|2)$ is the extension of $\mK$ over $\TKK(K)$ corresponding to the morphism $\mK\to \mathfrak{sl}_{2}$, where $1\in\mK$ is mapped to a semisimple element of $\mathfrak{sl}_{2}$.
\item $\Kan(K)$ is the extension of $\mK$ over $\TKK(K)$ corresponding to the morphism $\mK\to \mathfrak{sl}_{2}$, where $1\in\mK$ is mapped to a nilpotent element of $\mathfrak{sl}_{2}$.
\end{itemize}
  
In particular we find that
$$\widetilde{\TKK}(K) \not\cong \Der(\TKK(K))\quad\mbox{ and }\quad \Kan(K)\not\cong \TKK(K).$$
This gives counterexamples to the statements in Propositions \ref{Theorem of outer derivations} and \ref{PropKan}, for non-unital Jordan superalgebras. By Remark~\ref{RemTits} and the above, we do have
$$\TT(K,\Inn(K),\mathfrak{sl}_{2})\;\cong\; \TT(K,\Der(K), \mathfrak{sl}_2)\;\cong\; \TKK(K)\;\cong\; \mathfrak{psl}(2|2).$$
For the $3$-grading on $\mathfrak{psl}(2|2)$ corresponding to the interpretation as $\TKK(K)$, the algebra $\fg=\Out(\mathfrak{psl}(2|2)\cong \mathfrak{sl}_{2}$ is $3$-graded where $\fg_i$ has dimension one for $i\in\{-1,0,1\}$. This is in sharp contrast with Lemma \ref{V cong Der1} for the unital case. By Proposition~\ref{proposition Jordan superpairs and outer derivations}, $\widetilde{\TKK}(K)$ is the subalgebra of $\Der(\TKK(K))$ where only the degree $0$ derivations are added to $\TKK(K)$. In the same way, $\Kan(K)$ is the subalgebra of $\Der(\TKK(K))$ where only the degree $1$ derivations are added to $\TKK(K)$.

\appendix

\section{The Lie superalgebras of type $A$, $P$, $Q$ and $H$}

Consider an algebraically closed field $\mathbb{K}$ of characteristic zero.
We quickly review the Lie superalgebras of type $A$, $P$, $Q,$ and $H$, as different notations appear in literature. Our nomenclature is based on \cite{CW}. See also \cite[Theorem 1.11]{CW} for the list of simple finite dimensional Lie superalgebras.

\subsection{Type A}The general linear superalgebra $\mathfrak{gl}(m|n)$ is defined as $\End(\mathbb{K}^{m|n})$, with multiplication given by the super commutator.
Define the supertrace for a matrix $A \in \mathfrak{gl}(m|n)$ as $\text{str} (A) := \sum_i (-1)^{\abs{i}} A_{ii}$, where $\abs{i}=\oa$ for $i\leq m$ and $\abs{i}=\ob$ for $m+1\leq i\leq m+n$.
 The special linear superalgebra is $$\mathfrak{sl}(m|n)=\{A\in \mathfrak{gl} (m|n)\,|\,\, \text{str} (A)=0\},$$
 If $m\not= n$ then $\mathfrak{sl}(m|n)$ is simple.
If $m=n$ then $\langle I_{2n}\rangle$, with $I_{2n}$ the identity matrix, is an ideal in $\mathfrak{sl}(n|n)$ and 
$$\mathfrak{psl}(n|n):=\mathfrak{sl}(n|n)/\langle I_{2n}\rangle$$ is simple for $n>1$.
Similarly, we set  
$$\mathfrak{pgl}(n|n):=\mathfrak{gl}(n|n)/\langle I_{2n}\rangle.$$

\subsection{Type P}The periplectic Lie superalgebra is the subalgebra of $\mathfrak{gl}(n|n)$ defined as
\[
\mathfrak{pe}(n):=\left\{ \left(\begin{array}{cc} a&b\\c&-a^t \end{array}\right) \Big|\ a,b,c \in \mK^{n\times n} \text{ with } b^t=b,\ c^t=-c \right\}.
\]
The special periplectic Lie superalgebra is defined as $$\mathfrak{spe}(n):=\{x\in \mathfrak{p}(n) | \tr(a)=0\}.$$ 
It is simple for $n\geq 3$.

\subsection{Type Q} The queer Lie superalgebra is the subalgebra of $\mathfrak{gl}(n|n)$ defined as $$\mathfrak{q}(n):=\left\{ \left(\begin{array}{cc} a&b\\b&a \end{array}\right)\Big|\ a,b\in \mK^{n\times n} \right\}.$$
Remark that $\text{str}(X)=0$ for all $X \in \mathfrak{q}(n)$. 
 The special queer Lie superalgebra is defined as $$\mathfrak{sq}(n):=\left\{ \left(\begin{array}{cc} a&b\\b&a \end{array}\right)\Big|\ a,b \in \mK^{n\times n}, \tr(b)=0 \right\}.$$ \\
The projective special queer Lie superalgebra is defined as $$\mathfrak{psq}(n):= \mathfrak{sq}(n)/\langle I_{2n}\rangle.$$
It is simple for $n\geq 3$.
We also define the projective queer Lie superalgebra as
\[ \mathfrak{pq}(n):= \mathfrak{q}(n)/\langle I_{2n}\rangle.\]

\subsection{Type H}
Let $\Lambda(n)$ be the exterior algebra generated by $\xi_1, \ldots , \xi_n$. The indeterminates hence satisfy \[
\xi_i \xi_j = -\xi_j \xi_i.
\] 
This is an associative superalgebra where the generators are odd, $\abs{\xi_i}=\ob$. We also consider a compatible $\mZ$-grading, by setting $\deg \xi_i=1$.
Denote by $W(n)$ the algebra of derivations of the associative superalgebra $\Lambda(n)$. The Lie superalgebra $W(n)$ is simple for $n\geq 2$.

On $\Lambda(n)$, we define the following Poisson superbracket
\[
\{ f,g\} :=(-1)^{\abs{f}}\left( \sum_{i=1}^{n-2} \partial_{\xi_i}f \partial_{\xi_i}g + \partial_{\xi_{n-1}}f \partial_{\xi_n}g+\partial_{\xi_{n}}f \partial_{\xi_{n-1}}g\right),
\]
for $f$ and $g$ in $\Lambda(n)$. Then $(\Lambda(n),\{\cdot,\cdot\})$ becomes a Lie superalgebra with ideal $\langle 1 \rangle$. Consider the following Lie superalgebras \[ \widetilde{H}(n) := \Lambda(n)/ \langle 1 \rangle \;\text{ and }\; H(n):= [\widetilde{H}(n), \widetilde{H}(n)].\]
Note that $ \widetilde{H}(n)=H(n)\oplus   \mK \xi_1\cdots \xi_n$ as super vector spaces. 
The Lie superalgebra $H(n)$ is simple for $n \geq 4$.
We can embed $H(n)$ and $\widetilde{H}(n)$ into $W(n)$, using
$ f\mapsto \{f,\cdot\}.$

Consider $C:=\sum_{i=1}^n \xi_i \partial_{\xi_i} \in W(n)$, then $ \mK C \ltimes \widetilde{H}(n)$ is naturally defined as a subalgebra of $W(n)$.

We also define the semidirect product $\widetilde{H}(n-2) \ltimes  \Lambda(n-2)$, where the action of $\widetilde{H}(n-2)$ on $\Lambda(n-2)$ is given by the Poisson superbracket on $\Lambda(n-2)$, while the bracket of $\Lambda(n-2)$ is trivial. We further introduce, $\mK C \ltimes \left(\widetilde{H}(n-2) \ltimes  \Lambda(n-2)\right)  $, where $C$ acts by $[C,f]=(\deg f -2)f$ for $f \in  \widetilde{H}(n-2)$ and by $[C,g]=\deg g $ for $ g \in \Lambda(n-2)$.

\vspace{3mm}
\noindent
{\bf Acknowledgment.}
SB is a PhD Fellow of the Research Foundation - Flanders (FWO). KC is supported by Australian Research Council Discover-Project Grant DP140103239 and a postdoctoral fellowship of the Research Foundation - Flanders (FWO).

The authors thank Hendrik De Bie, Tom De Medts and Erhard Neher for helpful discussions and comments.

\noindent
SB: Department of Mathematical Analysis, Faculty of Engineering and Architecture, Ghent University, Krijgslaan 281, 9000 Gent, Belgium;
E-mail: {\tt Sigiswald.Barbier@UGent.be}

\vspace{2mm}

\noindent
KC: School of Mathematics and Statistics, University of Sydney, NSW 2006, Australia;
E-mail: {\tt kevin.coulembier@sydney.edu.au}

\date{}

\end{document}